\documentclass{amsart}
\usepackage{graphicx}

\newtheorem{thm}{Theorem}[section]
\newtheorem{cor}{Corollary}[section]
\newtheorem{lem}{Lemma}[section]

\theoremstyle{definition}
\newtheorem{defn}{Definition}[section]
\theoremstyle{remark}
\newtheorem{rem}{Remark}[section]
\numberwithin{equation}{section}

\begin{document}

\title[On the generalized Wiener bounded variation spaces with $p$-variable]{On the generalized Wiener bounded variation spaces with $p$-variable}%
\author{Ani Ozbetelashvili}
\address{Faculty of Exact and Natural Sciences, Ivane Javakhishvili Tbilisi State University, 13, University St., Tbilisi, 0143, Georgia}
\email{anaozbetelashvili5677@gmail.com}
\author{Shalva Zviadadze}
\address{Faculty of Exact and Natural Sciences, Ivane Javakhishvili Tbilisi State University, 13, University St., Tbilisi, 0143, Georgia}
\email{sh.zviadadze@gmail.com}
\thanks{This work was supported by Shota Rustaveli National Science Foundation  of Georgia (SRNSFG) grant no.: YS-19-480}%

\begin{abstract}
In this paper, we have investigated the generalized Wiener space of bounded variation with $p$-variable. Various results are obtained such as uniform convexity and reflexivity, there was characterized the set of points of discontinuity of functions from this space. For bounded exponents, it is shown the existence of right and left-hand limits in each point. Also, there is an unbounded exponent such that in corresponding generalized Wiener bounded variation space exists a function that does not have the right-hand limit at a point. Also, for some wide class of exponents in this space additivity of the variation is not fulfilled.
\end{abstract}
\maketitle

\section{introduction}
\par In 1881  C. Jordan \cite{jordan} introduced the notion of a function of bounded variation and established the relationship between those functions and monotonic ones when he was studying the convergence of the Fourier series. Later, problems in such areas as the calculus of variations, the convergence of Fourier series, geometric measure theory, mathematical physics, etc. motivated mathematicians to generalize the concept of bounded variation in various directions. Because of numerous authors who have generalized the notion of bounded variation, we do not list them.

\par The topic of function spaces with a variable exponent is a lively area of research at present, mainly due to its wide applications in the modeling of electrorheological fluids, in the study of image processing, and differential equations with non-standard growth. These aforementioned applications are mainly in Lebesgue and Sobolev spaces with variable exponent. Different aspects concerning these spaces with variable exponents can be found in \cite{uribe-fiorenza} and \cite{DHHR} and in their bibliography.

\par To shorten the expressions that will appear in further reasoning, let introduce some notations. Denote by $P:=\{Q_k\}_{k=1}^n$ the partition of $[a;b]$ where $Q_k=[t_{k-1};t_k]$ and $a=t_0<...<t_n=b$. Besides let $f(Q_k):=f(t_k)-f(t_{k-1})$ and $\Pi[a;b]$ denotes the set of all finite partitions of $[a;b]$.

\par The first attempt to investigate bounded variation spaces with variable exponent was due by H. Herda \cite{herda}. Following H. Nakano, he generalized the Wiener's $p$-th variation in the same way as $L^{p(\cdot)}$ space generalizes classical $L^p$ space\footnote{The so-called variable exponent Lebesgue space $L^{p(\cdot)}$ first was introduced by W. Orlicz \cite{orlicz}}. Namely, he considered two-variable exponent $p:\Omega\to[1;+\infty)$ where $\Omega = \{(x,y)\:|\:a\le y<x\le b\}$ and defined functional
$$
V_p^H(f):=\sup_{\Pi[a;b]}\sum_{k=1}^n|f(Q_k)|^{p(t_k,t_{k-1})},
$$
which is the Nakano modular on the space
$$
B_{p(x,y)}:= \left\{f:[a;b]\to\mathbb{R}\:|\:f(a)=0,\:V_p^H(f)<+\infty\right\}.
$$
H. Herda established various properties of such spaces, namely modular completeness, uniform convexity, and reflexivity.

\par Recently in \cite{castilo-merentes-rafeiro} authors introduced the space of functions of bounded variation with $p$-variable $BV^{p(\cdot)}$. They considered exponent $p:[a;b]\to(1;+\infty)$ such that $\sup p(x)<+\infty$ and for such exponent defined the functional by
$$
V_{[a;b]}^{p(\cdot)}(f):=\sup_{\Pi^*[a;b]}\sum_{k=1}^n|f(Q_k)|^{p(x_k)},
$$
where $\Pi^*[a;b]$ is a set of tagged partitions $P^*$ of $[a;b]$, i.e., a partition of the segment $[a;b]$ together with a finite sequence of numbers $x_1,...,x_{n}$ subject to the conditions that for each $k$, $t_{k-1}\le x_k\le t_k$.
\par  In \cite{castilo-merentes-rafeiro} authors define the space of functions of $p(\cdot)$-bounded variation by 
$$
BV^{p(\cdot)}[a;b]=\{f:[a;b]\to\mathbb{R}\:|\:f(a)=0,\:||f||_{BV^{p(\cdot)}}<+\infty\}, 
$$
where 
$$
||f||_{BV^{p(\cdot)}} =\inf\left\{\lambda>0\:|\:V_{[a;b]}^{p(\cdot)}(f/\lambda)\le1\right\} 
$$
is the Luxemburg norm. Authors of the paper \cite{castilo-merentes-rafeiro} investigated properties of the $BV^{p(\cdot)}[a;b]$ space. Namely, they proved that $BV^{p(\cdot)}[a;b]$ is Banach's space, if $q(x)\ge p(x)$ for all $x\in[a;b]$ then $BV^{p(\cdot)}[a;b]\hookrightarrow BV^{q(\cdot)}[a;b]$, they proved a Helly's principle of choice type result in $BV^{p(\cdot)}[a;b]$. Also, there is defined the analog of absolute p-continuous functions in the framework of variable space and proved that this space of absolutely continuous functions is closed and separable subspace of $BV^{p(\cdot)}[a;b]$. For the different results concerning the space $BV^{p(\cdot)}[a;b]$, and other spaces of bounded variation with $p$-variable see \cite{gimenez-mejia-merentes-rodriguez}, \cite{mejia-merentes-sanches}, \cite{mejia-merentes-sanches-valera}.
\par In the present paper, we are going to introduce the notion of bounded variation with $p$-variable differently and investigate obtained $WBV_{p(\cdot)}$ space of functions.
\par The first motivation to introduce bounded variation with $p$-variable in another way is that space $ BV^{p(\cdot)}[a;b]$ is not "stable" concerning changes in the exponent even at one point. In other words, if we change the exponent function at one point, we change the corresponding space. Therefore, in our definition, we will use the mean values of the exponent at the partition intervals. Such a definition of the bounded variation allows us to get a more "stable" space concerning changes in the exponent on the set of measure zero. Besides, below we will show that $BV^{p(\cdot)}[a;b]\subset WBV_{p(\cdot)}[a;b]$ for all exponents $p$, also we will give some example of an exponent for which $WBV_{p(\cdot)}[a;b]\backslash BV^{p(\cdot)}[a;b]\ne\emptyset$, this effect is achieved due to the fact that the constant exponent changes on a set of measure zero (see Theorem \ref{thm_anti_chadgma}).
\par Another motivation to consider the mean values of the exponent is that in 2016 independently were appeared two papers \cite{castilo-guzman-rafeiro} and \cite{kakochashvili-zviadadze} in which the Riesz bounded variation was introduced with variable exponent. Besides, the main result in these papers was a generalization of the Riesz result (about representation as to the indefinite integral of a function from the $L^p$ space) for a variable exponent Lebesgue spaces. In the \cite{castilo-guzman-rafeiro} authors considered so-called tagged partition as above in the definition of the space $BV^{p(\cdot)}[a;b]$ and they proved the corresponding result using log-Holder continuity of the exponent. On the other hand, in \cite{kakochashvili-zviadadze} authors introduced Riesz bounded variation with $p$-variable using mean values of the exponent which allowed them to prove the mentioned result for a much wider set of exponents such that the Hardy-Littlewood maximal operator is bounded on the $L^{p(\cdot)}$ space. 
\par These reasons motivated us to consider $WBV_{p(\cdot)}[a;b]$ space and investigate various properties of this space. Besides, we obtained some new results which are not considered in the papers published before.

\section{The space $WBV_{p(\cdot)}$ of functions of bounded variation}

\par Let $p:[a;b]\to [1;+\infty)$ be Lebesgue integrable function on $[a;b]$ and for interval $Q\subset [a;b]$ define
$$
\bar{p}(Q):=\left(\frac{1}{|Q|}\int_{Q}\frac{1}{p(x)}dx\right)^{-1}.
$$
Here and through the whole paper for the set $A$ the symbol $|A|$ means Lebesgue measure of the set $A$ and $\chi_A$ means the characteristic function of the set $A$. We again note that using $\bar{p}(Q)$ numbers in the exponent authors in \cite{kakochashvili-zviadadze} fully described variable exponent Sobolev space $W^{1,p(\cdot)}$ for the exponent functions for which Hardy-Littlewood maximal operator is bounded on the $L^{p(\cdot)}$.
\begin{defn}
Let consider functional 
\begin{equation}
\label{defn_wbv_space}
V_a^b(p,f):=\sup_{\Pi[a;b]}\sum|f(Q_{k})|^{\bar{p}(Q_{k})}<+\infty.
\end{equation}
We define the generalized Wiener bounded variation spaces with $p$-variable by
$$
WBV_{p(\cdot)}[a;b]=\{f:[a;b]\to\mathbb{R}\:|\:f(a)=0,\:||f||_{WBV_{p(\cdot)}}<+\infty\}, 
$$
where 
$$
||f||_{WBV_{p(\cdot)}} =\inf\left\{\lambda>0\:|\:V_a^b(p,f/\lambda)\le1\right\} 
$$
is the Luxemburg norm.
\end{defn}

\begin{rem}
\label{rem_wbv_gansazgvrebaze_vb_mimartebashi}
Let us note that in this definition exponent $p$ may be unbounded.
\end{rem}

\par Now we give reasoning which allows us to obtain some interesting results using mentioned results of H. Herda.
\par It is easy to see that each interval $Q\subset[a;b]$ determines the point in triangle $\Omega = \{(x,y)\:|\:a\le y<x\le b\}$, so, for given exponent $p(\cdot):[a;b]\to[1;+\infty)$ we construct two-variable exponent function $p(\cdot,\cdot):\Omega\to[1;\infty)$ by the following way $p(y,x)=\bar{p}(Q)$, where $Q=(x,y)$. Now using Herda's definition of the modular (see previous section) we obtain that $V_p^H(f)=V_a^b(p,f)$ for all functions $f$. Thus the results obtained by Herda are true for $WBV_{p(\cdot)}$ space and we will state them without proof.
\par By \cite[Theorem 6, Theorem 7, Theorem 8 and Theorem 9]{herda} we have following
\begin{thm}
The space $WBV_{p(\cdot)}$ is a Banach space.
\end{thm}
\begin{thm}
If the exponent $p:[a;b]\to[1;+\infty)$ is such that $1<\inf \bar{p}(Q)\le \sup\bar{p}(Q)<+\infty$, then the space $WBV_{p(\cdot)}$ is uniformly convex and reflexive.
\end{thm}

\begin{rem}
Let us note that condition $1<\inf \bar{p}(Q)\le \sup\bar{p}(Q)<+\infty$ is equivalent to condition $1<\mathop{\mbox{essinf}} p(x)\le \mathop{\mbox{esssup}} p(x)<+\infty$.
\end{rem}

\par Further, we will show other properties of the functions from $WBV_{p(\cdot)}$ space which can't be derived from previous results.

\begin{thm}
\label{thm_wbv_funqciebis_shemosazgvruloba}
If $f\in WBV_{p(\cdot)}$, then $f$ is bounded.
\end{thm}

\begin{proof}
Suppose the opposite and without of loss of generality assume that
$$
\limsup_{x\to b-}f(x)=+\infty
$$
and $f(x)<+\infty$ for some $x\in[a;b)$ (at least $f(a)=0$). Then there exists a sequence $\{x_n\}_{n\in\mathbb{N}}$ such that $x_n>x$ for all $n\in\mathbb{N}$, $x_{n}\uparrow b$ and $f(x_{n})\to+\infty$, $n\to+\infty$. Since $\bar{p}(Q)\ge1$ for all intervals $Q\subset[a;b]$ we obtain
$$
V_a^b(p,f)\ge|f(x_{n})-f(x)|^{\bar{p}([x;x_{n}])}\to+\infty, \:\: n\to+\infty.
$$
We get a contradiction $V_a^b(p,f)=+\infty$.
\end{proof}

\par Next we will investigate the relationship between $BV^{p(\cdot)}$ and $WBV_{p(\cdot)}$ spaces. Let start with the following

\begin{lem}
\label{lemma_sashualoze_didi_da_patara_mnishvnelobebi}
Given Lebesgue integrable function $f$, then there exist $x,y\in[a;b]$ such that
$$
f(x)\le \frac{1}{b-a}\int_{a}^{b}f(t)dt\le f(y).
$$
\end{lem}

\begin{proof} If $f$ is almost everywhere equal to the constant then this inequality will hold for almost all points $x,y\in[a;b]$. Assume that $f$ is not constant, we will show only the right-hand side inequality, the left-hand side inequality will be obtained analogously. Suppose the opposite, let for all $y\in (a;b)$ we have
$$
f(y)<\frac{1}{b-a}\int_{a}^{b}f(t)dt=\mu.
$$
We get that $f(y)<\mu$ for all points $y\in(a;b)$, if we integrate both sides of the inequality we get
$$
\int_{a}^{b}f(y)dy<\mu(b-a)=\int_a^bf(t)dt,
$$
which is a contradiction.
\end{proof}

\begin{thm}
\label{thm_chadgma_BV_WBV_shi}
Suppose, $p:[a;b]\to[1;+\infty)$ is bounded, then $BV^{p(\cdot)}\subset WBV_{p(\cdot)}$.
\end{thm}

\begin{proof}
Consider the partition $P$ of $[a;b]$ segment. If $p$ is a constant in the $Q_k$ interval, then $\bar{p}(Q_{k})=p(x_k)$ for all points $x_k\in Q_{k}$. If $p$ function is not constant on $Q_k$, then by Lemma \ref{lemma_sashualoze_didi_da_patara_mnishvnelobebi} there are $x_k^1, x_k^2\in Q_{k}$ points such that $p(x_k^1)\le\bar{p}(Q_{k})\le p(x_k^2)$. For the same partition $P$, consider the sum 
$$
\sum|f(Q_{k})|^{\bar{p}(Q_k)}\le\sum|f(Q_{k})|^{p(x_k^i)},
$$
where 
$$
x_{k-1}^i=
\begin{cases}
x_k^1, & \mbox{if } |f(Q_{k})|<1; \\
x_k^2, & \mbox{if } |f(Q_{k})|\ge 1. 
\end{cases}
$$
Since the partition was chosen arbitrarily we get 
$$
V_a^b(p,f)\le\sup_{\Pi^*[a;b]}\sum|f(Q_{k})|^{p(x_{k-1}^i)}.
$$
Thus, if $f\in BV_{p(\cdot)}$, then $V_a^b(f)<+\infty$ which means that $f\in WBV_{p(\cdot)}$.
\end{proof}

\begin{cor}
Let $p:[a;b]\to[1;+\infty)$ be a bounded function. Then $WBV_{p(\cdot)}$ has a sub-space isomorphic to $c_0$.
\end{cor}

\begin{proof}
See \cite[Corollary 3.3.]{castilo-merentes-rafeiro}.
\end{proof}

Naturally, there raises a question: Does the inclusion $WBV_{p(\cdot)}\subset BV^{p(\cdot)}$ hold for all bounded exponents? The answer to this question is negative.

\begin{thm}
\label{thm_anti_chadgma}
There exists a bounded exponent for which $WBV_{p(\cdot)}\backslash BV^{p(\cdot)}\ne\emptyset$.
\end{thm}

\begin{proof}
Consider [0;1] segment, the set $A:=\left\{1/n\:|\: n\in\mathbb{N}\right\}$, the exponent function $p(x)=4-2\chi_{A}(x)$ and
$$
f(x)=\sqrt{x}\cdot\chi_A(x).
$$
It is clear that $f(0)=0$. Let's consider intervals $\left[\frac{1}{k};\frac{2k-1}{2k^2-2k}\right]$, $k\in\{2,3,...\}$, it is obvious that for any natural number $N>2$ we have
\begin{align*}
&
\sup_{\Pi^*[a;b]}\sum|f(Q_{k})|^{p(x_{k-1})}\ge
\sum_{k=2}^{N}\left|f\left(\left[\frac{1}{k};\frac{2k-1}{2k^2-2k}\right]\right)\right|^{p\left(1/k\right)}
\\&\qquad
=\sum_{k=2}^{N}\frac{1}{k}\to+\infty,\quad N\to+\infty.
\end{align*}
We obtain that $f\notin BV^{p(\cdot)}$. On the other hand since $|A|=0$, for arbitrary partition $P$, we have $\bar{p}(Q)=4$ for all $Q\in P$. Then
$$
\sup_{\Pi[a;b]}\sum|f(Q_{k})|^{\bar{p}(Q_{k})}\le2\sum_{k=1}^{+\infty}\left(\frac{1}{\sqrt{k}}\right)^4=\frac{\pi^2}{3}.
$$
Thus $f\in WBV_{p(\cdot)}\backslash BV^{p(\cdot)}$.
\end{proof}

\begin{thm}
The space $WBV_{p(\cdot)}$ is not separable.
\end{thm}

\begin{proof}
Consider the set of the functions $A:=\{\chi_{\{x\}}\:|\:x\in(a;b]\}$. It is clear that $A\subset WBV_{p(\cdot)}$ and $||f(\cdot)-g(\cdot)||_{WBV_{p(\cdot)}}\ge1$, for all $f\ne g$, $f,g\in A$. This implies that $B(f,1/2)\cap B(g,1/2)=\emptyset$, for all $f\ne g$, $f,g\in A$, where $B(f,1/2)$ denotes the open ball centered in $f$ and with a radius of 1/2. Now, since every dense subset of $WBV_{p(\cdot)}$ must have at least a point inside each $B(f,1/2)$, $f\in A$, and the set $A$ has the power of the continuum, therefore there is no countable set that is dense in $WBV_{p(\cdot)}$.
\end{proof}

\begin{thm}
\par Let given functions $p_1,p_2:[a;b]\to[1;+\infty)$ and $p_1(x)\le p_2(x)$ a.e. $x\in[a;b]$, then
$WBV_{p_1(\cdot)}\hookrightarrow WBV_{p_2(\cdot)}$.
\end{thm}

\begin{proof}
Since $p_1(x)\le p_2(x)$ we get 
$$
\frac{1}{p_1(x)}\ge\frac{1}{p_2(x)},
$$ 
now by monotonicity of the integral for any interval $Q$ we obtain 
$$
\frac{1}{|Q|}\int_{Q}\frac{1}{p_1(x)}dx\ge\frac{1}{|Q|}\int_{Q}\frac{1}{p_2(x)}dx.
$$ 
Finally $\bar{p}_1(Q)\le\bar{p}_2(Q)$.

\par Consider the function $f$ such that $\left\|f\right\|_{WBV_{p_1(\cdot)}}=1$, then for all intervals $Q\subset[a;b]$ we have $|f(Q)|\le 1$ and by inequality $\bar{p}_1(Q)\le\bar{p}_2(Q)$ we get
$$
\sum|f(Q)|^{\bar{p}_2(Q)}\le\sum|f(Q)|^{\bar{p}_1(Q)}.
$$
If we write supremum on both sides of the inequality we will conclude that
$$
V_a^b(p_2,f)\le V_a^b(p_1,f).
$$
Which implies that 
$$
||f||_{WBV_{p_2(\cdot)}}\le||f||_{WBV_{p_1(\cdot)}}.
$$
\end{proof}

\par The next two theorems concern the existence of one-sided limits at each point of a function from $WBV_{p(\cdot)}$ and how the existence of one-sided limits depends on the exponent $p$.

\begin{thm}
\label{thm_calmxrivi_zgvrebis_arseboba}
Let $p:[a;b]\to[1;+\infty)$ be essentially bounded function and $f\in WBV_{p(\cdot)}$, then at each point of  the domain of function $f$ there exist left-hand and right-hand limits $($at the end points of the interval we consider only $f(a+)$ and $f(b-)$$)$.
\end{thm}

\begin{proof}
As usual let the symbols $f(x-)$ and $f(x+)$ stand for a left-hand and right-hand limits of the function $f$ at the point $x$. 
\par First, we prove the existence of the $f(x-)$. By the Theorem \ref{thm_wbv_funqciebis_shemosazgvruloba} the function $f$ is bounded. Let suppose the opposite that there exists a point $x\in(a;b]$ such that
$$
-\infty<c:=\liminf_{t\to x-} f(t)<\limsup_{t\to x-}f(t)=:d<+\infty.
$$
Then there exist the sequences $\{t_n\}_{n\in\mathbb{N}}$ and $\{s_n\}_{n\in\mathbb{N}}$ such that $t_n<s_n<t_{n+1}<s_{n+1}$ for all $n\in\mathbb{N}$, also $t_n\uparrow x$, $s_n\uparrow x$, $n\to+\infty$ and
$$
c=\lim_{n\to+\infty}f(s_n),\qquad d=\lim_{n\to+\infty}f(t_n).
$$
The last implies that there exists a natural number $N$ such that for all natural numbers $n>N$ we have
$$
f(t_n)-f(s_n)>\frac{d-c}{2}.
$$
Then for all natural $k$
\begin{align*}
V_a^b(p,f)&\ge \sum_{n=N+1}^{N+k}\left(f(t_n)-f(s_n)\right)^{\bar{p}([t_n;s_n])}\ge \sum_{n=N+1}^{N+k}\left(\frac{d-c}{2}\right)^{\bar{p}([t_n;s_n])}
\\&
\ge \left\{
\begin{array}{lcc}
\sum\limits_{n=N+1}^{N+k}\left(\frac{d-c}{2}\right)^{||p||_\infty} & \text{if} & \frac{d-c}{2}<1\\
\sum\limits_{n=N+1}^{N+k}\left(\frac{d-c}{2}\right) & \text{if} & \frac{d-c}{2}\ge1
\end{array}
\right\}=B(k).
\end{align*}
It is easy to see that 
$$
\lim_{k\to+\infty}B(k)=+\infty.
$$
Therefore $V_a^b(p,f)=+\infty$ which is contradiction. Similarly, we obtain the existence of $f(x+)$.
\end{proof}

\par One of the direct consequences of the Theorem \ref{thm_calmxrivi_zgvrebis_arseboba} is following

\begin{cor}
Let $p:[a;b]\to[1;+\infty)$ be essentially bounded function. Then if function $f\in WBV_{p(\cdot)}$ has a point of discontinuity it should be only the first kind.
\end{cor}

\begin{cor}
If $f\in BV^{p(\cdot)}$, then for all points $x\in[a;b]$ there exist $f(x-)$ and $f(x+)$. Therefore if function $f$ has a point of discontinuity it should be only the first kind.
\end{cor}

\begin{proof}
Since, $BV^{p(\cdot)}$ is defined only for bounded exponents we conclude that $p$ is bounded. Therefore by Theorem \ref{thm_calmxrivi_zgvrebis_arseboba} and Theorem \ref{thm_chadgma_BV_WBV_shi} we obtain the desired result.
\end{proof}

\par Naturally, there raises a question: Does the same result maintain if the exponent is not essentially bounded? As it turned out, the answer to this question is negative.

\begin{thm}
There exists essentially unbounded exponent $p:[a;b]\to[1;+\infty)$ and function $f\in WBV_{p(\cdot)}$ such that $f(a+)$ does not exists.
\end{thm}
\begin{proof}
Without loss of generality assume that $[a;b]=[0;1]$. Let $A=\{1/n\:|\:n\in\mathbb{N}\}$, $f(x)=\chi_A(x)/2$ and
$$
p(x):=1+\sum_{n\in\mathbb{N}}n\cdot\chi_{\left(\frac{1}{n+1};\frac{1}{n}\right]}(x).
$$
It is easy to see that
$$
V_a^b(p,f)\le 2\cdot\sum_{n\in\mathbb{N}}\frac{1}{2^n}=2,
$$
so, $f\in WBV_{p(\cdot)}$, but $f(a+)$ does not exists.
\end{proof}

\par As we see for essentially unbounded exponents in corresponding $WBV_{p(\cdot)}$ space may exist a function with a second kind discontinuity at a point.

\par It is well known that the set of discontinuity points of the function with bounded variation in Jordan's sense at most is countable. This proposition stems from the fact that a function of bounded variation in Jordan's sense is represented as the difference of two non-decreasing functions, while the set of discontinuity points of a monotonic function is at most countable. It is clear that in general the function from space $WBV_{p(\cdot)}$ is not represented as the difference of the two non-decreasing functions. Nevertheless, a similar result is still fair for the functions of this class.

\begin{thm}
\label{thm_ckvetis_certilta_tvladoba}
Let $p:[a;b]\to[1;+\infty)$ be essentially bounded function, then the set of discontinuous points of the function $f\in WBV_{p(\cdot)}$ is at most countable.
\end{thm}

\begin{proof}
Suppose the opposite that the set of discontinuity points $E\subset[a;b]$ of the function $f\in WBV_{p(\cdot)}$ is not countable. By virtue of the Theorem \ref{thm_calmxrivi_zgvrebis_arseboba}, there exist one-sided limits at each point, including the points of the set $E$. Let $E_n\subset E$ be the set of such points for which hold inequalities $|f(x)-f(x-)|>1/n$ or $|f(x)-f(x+)|>1/n$. Since $E$ is not countable there exists a number $n\in\mathbb{N}$ such that the set $E_n$ is not finite. Indeed, if for all $n\in\mathbb{N}$ the set $E_n$ is finite, then we obtain that the set $E$ is a countable union of the finite sets, therefore the set $E$ is at most countable and we get the contradiction. So, the set $E_n$ is infinite. Since any infinite set contains a countable subset, without loss of generality assume that the set $E_n$ is countable. For any number $M>0$ there exists $N\in\mathbb{N}$ such that $N>n^{||p||_\infty}\cdot M$. From the set $E_n$ let us choose finite and increasing sequence of points $\{x_k\}_{k=0}^N$ and define
$$
\delta_N:=\frac{1}{2}\cdot\min_{k\in\{1,...,N\}}(x_k-x_{k-1}).
$$
For each member $x_k$ of this sequence choose point $x'_k\in(x_k-\delta_N;x_k+\delta_N)$ such that $|f(x_k)-f(x'_k)|>1/n$. By $Q'_k$ symbol, we denote interval $[x'_k;x_k]$ or $[x_k;x'_k]$. It is obvious that the intervals $Q'_k$ are disjoint. Consider some finite partition $P$ of the segment $[a;b]$ such that $\{Q'_k\}_{k=0}^N\subset P =\{Q_k\}_{k=0}^m $, then
\begin{align*}
V_a^b(p,f)&\ge \sum_{s=0}^m|f(Q_s)|^{\bar{p}(Q_s)}\ge \sum_{k=0}^{N-1}|f(Q'_k)|^{\bar{p}(Q'_k)}
\\&
\ge  \sum_{k=0}^{N-1}\left(\frac{1}{n}\right)^{\bar{p}(Q'_k)}\ge N\cdot \left(\frac{1}{n}\right)^{||p||_\infty}>M.
\end{align*}
Since the number $M$ was chosen arbitrarily we obtain that $V_a^b(p,f)=+\infty$, i.e. $f\notin WBV_{p(\cdot)}$.
\end{proof}

\begin{cor}
If $f\in BV^{p(\cdot)}$, then the set of discontinuity points of $f$ is at most countable.
\end{cor}

\begin{proof}
By definition of the space $BV^{p(\cdot)}$, Theorem \ref{thm_chadgma_BV_WBV_shi} and Theorem \ref{thm_ckvetis_certilta_tvladoba} we easily obtain the desired result.
\end{proof}

\par As above analogously here raise questions: What role plays the boundedness of the exponent in the previous theorem? Does the same result maintain if the exponent is not essentially bounded? As it turned out, the answer to the last question is negative.

\begin{thm}
There is essentially unbounded exponent $p:[a;b]\to[1;+\infty)$ and $f\in WBV_{p(\cdot)}$ such that the set of discontinuity points of $f$ has continuum power.
\end{thm}

\begin{proof}
Without loss of generality suppose that $a=0$ and $b=1$. Let $\mathcal{C}$ denotes Cantor's set without point 0 and $f(x)=\frac{1}{2}\cdot\chi_{\mathcal{C}}(x)$. It is clear that for the given function the set $\mathcal{C}$ is a set of discontinuity points. Let's enumerate the contiguous intervals of the set $\mathcal{C}$ in the following way $I_1^1=(1/3;2/3)$, $I_2^1=(1/9;2/9)$, $I_2^2=(7/9;8/9)$, . . ., then each contiguous interval of the set $\mathcal{C}$ is denoted by $I_n^k$ symbol, where $k\in\{1,...,2^{n-1}\}$. Now we define the function in the exponent 
$$
p(x)=\chi_{\mathcal{C}}(x)+2\cdot\sum_{n=1}^{+\infty}n\cdot\sum_{k=1}^{2^{n-1}}\chi_{I_n^k}(x).
$$
\par It is remains to show that $f\in WBV_{p(\cdot)}$. Consider any finite partition $P=\{Q_s\}_{s=0}^m$ of the segment $[0;1]$. It is clear that for any $Q_s$ interval $f(Q_s)=0$ or $|f(Q_s)|=1/2$. If the endpoints of the interval $Q_s$ belong to the set $\mathcal{C}$ or the complement of the set $\mathcal{C}$ then $f(Q_s)=0$. Since $Q_s$ are disjoint intervals, then each interval $I_n^k$ intersects with at most two such intervals $Q_s$ that $|Q_s|\le |I_n^k|$ and $|f(Q_s)|=1/2$. So, for any partition of the segment $[0;1]$ we have
$$
\sum_{s=0}^m|f(Q_s)|^{\bar{p}(Q_s)}\le 2\cdot\sum_{n=1}^{+\infty}\frac{1}{2^n}=2.
$$
By the fact that partition was chosen arbitrarily, we obtain that $V_a^b(p,f)\le2$, i.e. $f\in WBV_{p(\cdot)}$.
\end{proof}

\par Now we will present the results concerning the additivity of the variation $V_a^b(p,\cdot)$.
\begin{thm}
\label{thm_anti_adiciuroba}
For any point $c\in(a;b)$ we have
$$
V_a^c(p,f)+V_c^b(p,f)\le V_a^b(p,f).
$$
\end{thm}
\begin{proof}
For any $\varepsilon>0$ there exist a partition of the segment $[a;c]$ by points $a=x_0<x_1<...<x_n=c$ and a partition of the segment $[c;b]$ by points $c=y_0<y_1<...<y_m=b$ such that
$$
\sum_{k=1}^{n}|f([x_{k-1};x_k])|^{\bar{p}([x_{k-1};x_k])}>V_a^c(p,f)-\varepsilon/2.
$$
$$
\sum_{k=1}^{m}|f([y_{k-1};y_k])|^{\bar{p}([y_{k-1};y_k])}>V_c^b(p,f)-\varepsilon/2.
$$
Union of points of these partitions gives the specific partition of the segment $[a;b]$ by points $z_k$, where $z_k=x_k$, $k\in\{0,...,n\}$ and $z_{n+k}=y_k$, $k\in\{0,...,m\}$ (specific because the point $c$ always belongs to such partition), then
\begin{align*}
&
V_a^c(p,f)+V_c^b(p,f)-\varepsilon
\\&\qquad
<\sum_{k=1}^{n}|f([x_{k-1};x_k])|^{\bar{p}([x_{k-1};x_k])}+\sum_{k=1}^{m}|f([y_{k-1};y_k])|^{\bar{p}([y_{k-1};y_k])}
\\&\qquad
=\sum_{k=1}^{n+m}|f([z_{k-1};z_k])|^{\bar{p}([z_{k-1};z_k])}\le V_a^b(p,f).
\end{align*}
Since $\varepsilon$ was chosen arbitrarily we obtain the desired result.
\end{proof}

\par The following result shows an even more contrasting property of $WBV_{p(\cdot)}$ space compared to the Jordan and Wiener bounded variation classes. In general, if the exponent $p$ is not equivalent to the constant, then there might not exist such number $C>0$ that the following inequality holds
$$
V_a^b(p,f)\le C\cdot\left(V_a^c(p,f)+V_c^b(p,f)\right)
$$
for all $c\in(a;b)$ and for all $f\in WBV_{p(\cdot)}$. 
\par Let's introduce a notation
$$
\bar{p}_-^x(a;b):=\inf\{\bar{p}([c;d])\:|\:x\in[c;d]\subset[a;b]\}.
$$

\begin{thm}
Let given function $p:[a;b]\to[1;+\infty)$. If there exists point $x\in(a;b)$ such that
\begin{equation}
\label{cond_adiciurobis_uarkofistvis}
\bar{p}_-^x(a;b)<\max\{\bar{p}_-^x(a;x),\bar{p}_-^x(x;b)\},
\end{equation}
then for any number $C\ge1$ there exists function $f:=f_C$ such that $f\in WBV_{p(\cdot)}$ and
$$
V_a^b(p,f)>C\cdot\left(V_a^x(p,f)+V_x^b(p,f)\right).
$$
\end{thm}
\begin{proof}
Let \eqref{cond_adiciurobis_uarkofistvis} holds, without loss of generality suppose that 
$$
\bar{p}_-^x(x;b)=\max\{\bar{p}_-^x(a;x),\bar{p}_-^x(x;b)\}
$$
and $A:=\bar{p}_-^x(x;b)-\bar{p}_-^x(a;b)$. Choose number $K>C^{\frac{1-A}{A}}$ such that $KC\ge1$ and define function $f(t)=\frac{1}{KC}\chi_{(x;b]}(t)$. It is easy to see that $f\in WBV_{p(\cdot)}$ besides
$$
V_a^b(p,f)=\left(\frac{1}{KC}\right)^{\bar{p}_-^x(a;b)},
$$
On the other hand $V_a^x(p,f)=0$ and
$$
V_x^b(p,f)=\left(\frac{1}{KC}\right)^{\bar{p}_-^x(x;b)}.
$$
Then
\begin{align*}
&
C\cdot\left(V_a^x(p,f)+V_x^b(p,f)\right)=C\cdot V_x^b(p,f)=
\\&\qquad
=C\cdot\left(\frac{1}{KC}\right)^{\bar{p}_-^x(x;b)}=C\cdot\left(\frac{1}{KC}\right)^{\bar{p}_-^x(x;b)-\bar{p}_-^x(a;b)+\bar{p}_-^x(a;b)}=
\\&\qquad
=C\cdot\left(\frac{1}{KC}\right)^{A}\cdot \left(\frac{1}{KC}\right)^{\bar{p}_-^x(a;b)}=\frac{C^{1-A}}{K^A}\cdot V_a^b(p,f).
\end{align*}
Since, $\frac{C^{1-A}}{K^A}<1$ we get that $V_a^b(p,f)>C\cdot\left(V_a^x(p,f)+V_x^b(p,f)\right)$.
\end{proof}

\par The condition \eqref{cond_adiciurobis_uarkofistvis} seems to be nontransparent and we are trying to fix this. Consider $\bar{p}_-^x(a;b)$ by its definition we get
\begin{align*}
&
\bar{p}_-^x(a;b)=\inf\left\{\bar{p}([c;d])\:|\:x\in[c;d]\subset[a;b]\right\}=
\\&\qquad
=\inf\left\{\left(\frac{1}{d-c}\int_c^d\frac{1}{p(t)}dt\right)^{-1}\:|\:x\in[c;d]\subset[a;b]\right\}=
\\&\qquad
=\left(\sup\left\{\frac{1}{d-c}\int_c^d\frac{1}{p(t)}dt\:|\:x\in[c;d]\subset[a;b]\right\}\right)^{-1}=\left(M\left(\frac{1}{p}\right)(x)\right)^{-1}.
\end{align*}
Here the symbol $M$ denotes the Hardy-Littlewood maximal operator. By the analogous reasoning we obtain that 
$$
\bar{p}_-^x(a;x)=\left(M^-\left(\frac{1}{p}\right)(x)\right)^{-1}\quad \text{and} \quad \bar{p}_-^x(x;b)=\left(M^+\left(\frac{1}{p}\right)(x)\right)^{-1}.
$$
Where $M^-$ and $M^+$ denote the one-sided Hardy-Littlewood maximal operators. So, in the terms of Hardy-Littlewood maximal operator, we can rewrite condition \eqref{cond_adiciurobis_uarkofistvis} in the following form
$$
M\left(\frac{1}{p}\right)(x)>\min\left\{M^-\left(\frac{1}{p}\right)(x),M^+\left(\frac{1}{p}\right)(x)\right\}.
$$
The investigation of the class of such exponents for which there exists a point $x\in(a;b)$ such that above inequality holds is beyond of the scope of our paper and let it be an open problem. A particular example of such exponent is $p(t)=10\cdot\chi_{[a;(a+b)/2]}(t)+2\cdot\chi_{((a+b)/2;b]}(t)$. The reader can easily construct other exponents for which the above inequality holds for some point $x\in(a;b)$. We believe that mentioned inequality holds for all functions which are not equivalent to the constant.
\par Next, we will set the stage for proving Helly's principle of choice type result.
\begin{lem}
\label{lemma_variaciit_gansaz_funqciis_monotonuroba}
Let $F(x):=V_a^x(p,f)$ then $F$ is non-decreasing function on $[a;b]$.
\end{lem}
\begin{proof}
Since, $F(a)=0$ consider points $a<x<y\le b$, then for the any partition of the segment $[a;x]$ by points $a=t_0<t_1<...<t_n=x$ we gave partition of the segment $[a;y]$ by points $a=t_0<t_1<...<t_n<t_{n+1}=y$. Then it is clear that for all functions $f$ we get
$$
\sum_{k=0}^n|f(Q_k)|^{\bar{p}(Q_k)}\le \sum_{k=0}^{n+1}|f(Q_k)|^{\bar{p}(Q_k)}.
$$
If we take supremum to both sides of the inequality with respect to the all finite partitions of the segment $[a;x]$ we obtain:
\begin{align*}
F(x)&=\sup_{\Pi_{[a;x]}}\sum_{k=0}^n|f(Q_k)|^{\bar{p}(Q_k)}
\\
&\le \sup_{\Pi_{[a;x]}}\sum_{k=0}^{n}|f(Q_k)|^{\bar{p}(Q_k)}+|f(x;y)|^{\bar{p}(x;y)}
\\
&\le \sup_{\Pi_{[a;y]}}\sum_{k=0}^{m}|f(Q_k)|^{\bar{p}(Q_k)}=F(y).
\end{align*}
\end{proof}

\begin{rem}
Let's note that the above lemma is true for all functions not only for functions of the space  $WBV_{p(\cdot)}$. 
\end{rem}

\par With the above Lemma \ref{lemma_variaciit_gansaz_funqciis_monotonuroba} taken into account, we can prove Helly's principle of choice type result in the space $WBV_{p(\cdot)}$. The scheme of the proof follows closely the one given by Musielak and Orlicz (see \cite{musielak-orlicz}).

\begin{thm}
\label{thm_helis_teorema}
Let $p:[a;b]\to[1;+\infty)$ be essentially bounded function and $\mathcal{F}$ is a uniformly bounded infinite family of the functions of the space $WBV_{p(\cdot)}$. If the variations of the functions from $\mathcal{F}$ are bounded by the same number, then from the family $\mathcal{F}$ we can choose the sequence which is pointwise convergent to a function $f\in WBV_{p(\cdot)}$.
\end{thm}

\begin{proof}
For all functions $f_\alpha\in\mathcal{F}$ define non-decreasing function $F_\alpha(x)=V_a^x(p,f_\alpha)$ (see Lemma \ref{lemma_variaciit_gansaz_funqciis_monotonuroba}). By Helly’s principle of choice for sequences of monotonic functions, we get that there exists a sequence of functions $\{F_n\}$ convergent to a non-decreasing function $F$ at every point $x$ of the interval $[a;b]$. The sequence $\{F_n\}$ uniquely defines the sequence $\{f_n\}$ of the functions from $\mathcal{F}$. By the Cantor diagonal method we extract from the sequence $\{f_n\}$ the subsequence $\{f_{n_k}\}$ which is convergent at each point of the set $E=(\mathbb{Q}\cap[a;b])\cup\{a,b\}$. Suppose $f_{n_k}\to f$, $k\to+\infty$ on the set $E$. Also it is clear that $F_{n_k}\to F$, $k\to+\infty$ on the segment $[a;b]$.
\par Assume that $x_0$ is irrational number at which $F$ is continuous. We will prove that the number sequence $f_{n_k}(x_0)$ is convergent. For any number $\varepsilon>0$ choose a rational number $t$ such that $x_0<t\in(a;b)$ and 
$$
0\le F(t)-F(x_0)<\varepsilon^{\bar{p}(x_0;t)}/3.
$$
Choose number $N_1\in\mathbb{N}$ such that for all $k>N_1
$ we have 
$$
|F_{n_k}(t)-F(t)|<\varepsilon^{\bar{p}(x_0;t)}/3 \quad \text{and} \quad |F_{n_k}(x_0)-F(x_0)|<\varepsilon^{\bar{p}(x_0;t)}/3, 
$$
then when $k>N_1$ the Theorem \ref{thm_anti_adiciuroba} gives
\begin{align*}
&|f_{n_k}(t)-f_{n_k}(x_0)|^{\bar{p}(x_0;t)}\le V_{x_0}^t(p,f_{n_k})
\\&\qquad
\le V_a^t(p,f_{n_k})-V_a^{x_0}(p,f_{n_k})=F_{n_k}(t)-F_{n_k}(x_0)
\\&\qquad
\le|F_{n_k}(t)-F(t)| + |F(t)-F(x_0)| + |F(x_0)-F_{n_k}(x_0)|<\varepsilon^{\bar{p}(x_0;t)}.
\end{align*}
We obtain that
$$
|f_{n_k}(t)-f_{n_k}(x_0)|<\varepsilon,
$$
Now choose number $N_2$ such that when $m,s>N_2$ we get
$$
|f_{n_m}(t)-f_{n_s}(t)|<\varepsilon.
$$
If $m,s>N:=\max\{N_1,N_2\}$ then
\begin{align*}
&|f_{n_m}(x_0)-f_{n_s}(x_0)|
\\&\qquad
=|f_{n_m}(x_0)-f_{n_m}(t)+f_{n_m}(t)-f_{n_s}(t)+f_{n_s}(t)-f_{n_s}(x_0)|
\\&\qquad
\le|f_{n_m}(x_0)-f_{n_m}(t)|+|f_{n_m}(t)-f_{n_s}(t)|+|f_{n_s}(t)-f_{n_s}(x_0)|<3\varepsilon.
\end{align*}
So, $f_{n_k}(x_0)$ is convergent in the points of continuity of the $F$. Since the set of discontinuity points of non-decreasing function $F$ is at most countable and the sequence of the functions $f_{n_k}$ is uniformly bounded using again the Cantor diagonal method we extract subsequence $\{f_{n_{k_s}}\}$ which is convergent at each point of the segment $[a;b]$. Also it is evident that $V_a^b(p,f)$ is a finite number. Indeed since for each number $s$ we have $V_a^b(p,f_{n_{k_s}})\le M$, this means that for all finite partitions of $[a;b]$ we have
$$
\sum_{i=1}^m|f_{n_{k_s}}(Q_i)|^{\bar{p}(Q_i)}\le M.
$$ 
Letting $s\to+\infty$, we obtain that
$$
\sum_{i=1}^m|f(Q_i)|^{\bar{p}(Q_i)}\le M.
$$
The last inequality holds for all finite partitions of the segment $[a;b]$, therefore $V_a^b(p,f)\le M$, i.e. $f\in WBV_{p(\cdot)}$.
\end{proof}

\end{document}